\documentclass[11pt]{article}

\usepackage[margin=1.2in]{geometry}

\usepackage{amsmath}
\usepackage{amsthm}
\usepackage{graphicx}
\usepackage{color}
\usepackage{amsfonts}
\usepackage{amssymb}
\usepackage{amscd}
\usepackage{url}

\newcommand{\re}{\mathbb{R}}

\newcommand{\N}{\mathbb{N}}

\newcommand{\diag}{\mbox{diag}}

\newcommand{\half}{\frac{1}{2}}
\newcommand{\lmd}{\lambda}

\newcommand{\nn}{\nonumber}

\newcommand{\dt}{\delta}
\newcommand{\Dt}{\Delta}
\def\af{\alpha}
\def\bt{\beta}

\newcommand{\sig}{\sigma}

\newcommand{\reff}[1]{(\ref{#1})}
\newcommand{\pt}{\partial}
\newcommand{\prm}{\prime}

\newcommand{\mc}[1]{\mathcal{#1}}

\newcommand{\bdes}{\begin{description}}
\newcommand{\edes}{\end{description}}

\newcommand{\bal}{\begin{align}}
\newcommand{\eal}{\end{align}}

\newcommand{\bnum}{\begin{enumerate}}
\newcommand{\enum}{\end{enumerate}}

\newcommand{\bit}{\begin{itemize}}
\newcommand{\eit}{\end{itemize}}

\newcommand{\bea}{\begin{eqnarray}}
\newcommand{\eea}{\end{eqnarray}}
\newcommand{\be}{\begin{equation}}
\newcommand{\ee}{\end{equation}}

\newcommand{\baray}{\begin{array}}
\newcommand{\earay}{\end{array}}

\newcommand{\bsry}{\begin{subarray}}
\newcommand{\esry}{\end{subarray}}

\newcommand{\bca}{\begin{cases}}
\newcommand{\eca}{\end{cases}}

\newcommand{\bcen}{\begin{center}}
\newcommand{\ecen}{\end{center}}

\newcommand{\bbm}{\begin{bmatrix}}
\newcommand{\ebm}{\end{bmatrix}}

\newcommand{\bmx}{\begin{matrix}}
\newcommand{\emx}{\end{matrix}}

\newcommand{\bpm}{\begin{pmatrix}}
\newcommand{\epm}{\end{pmatrix}}

\newcommand{\btab}{\begin{tabular}}
\newcommand{\etab}{\end{tabular}}

\newtheorem{theorem}{Theorem}[section]
\newtheorem{pro}[theorem]{Proposition}

\newtheorem{lem}[theorem]{Lemma}

\newtheorem{cor}[theorem]{Corollary}

\newtheorem{ass}[theorem]{Assumption}

\theoremstyle{definition}

\newtheorem{exm}[theorem]{Example}

\renewcommand{\subsection}[1]{
    \stepcounter{subsection}
    \settowidth{\hangindent}{\bf\thesubsection.~}
    \hangafter=1
    \bigskip\bigskip\noindent
    {\bf\hbox{\thesubsection.~}#1}\par
    \nobreak
    \medskip
}

\begin{document}

\title{Polynomial Matrix Inequality and Semidefinite Representation
\author{Jiawang Nie\footnote{Department of Mathematics,
University of California, 9500 Gilman Drive, La Jolla, CA 92093.
Email: njw@math.ucsd.edu. The research was partially supported by NSF grants
DMS-0757212, DMS-0844775 and Hellman Foundation Fellowship.}}
\date{March 25, 2011}
}

\maketitle

\begin{abstract}
Consider a convex set $S=\{x \in \mc{D}: G(x) \succeq 0\}$ where
$G(x)$ is a symmetric matrix
whose every entry is a polynomial or rational function,
$\mc{D} \subseteq \re^n$ is a domain on which $G(x)$ is defined,
and $G(x)\succeq 0$ means $G(x)$ is positive semidefinite.
The set $S$ is called semidefinite representable if
it equals the projection of a higher dimensional set
which is defined by a linear matrix inequality (LMI).
This paper studies sufficient conditions guaranteeing
semidefinite representability of $S$.
We prove that $S$ is semidefinite representable in the following cases:
(i) $\mc{D} = \re^n$, $G(x)$ is a matrix polynomial and matrix sos-concave;
(ii) $\mc{D}$ is compact convex, $G(x)$ is a matrix polynomial
and strictly matrix concave on $\mc{D}$;
(iii) $G(x)$ is a matrix rational function
and q-module matrix concave on $\mc{D}$.
Explicit constructions of semidefinite representations are given.
Some examples are illustrated.
\end{abstract}

\normalsize

\section{Introduction}

Suppose $S$ is a convex set in $\re^n$ given in the form
\be \label{eq:Sdef}
S = \{x\in \mc{D}:\, G(x) \succeq 0 \}.
\ee
Here $\mc{D}\subseteq \re^n$ is a domain, and
$G(x)$ is a $m\times m$ symmetric matrix polynomial, that is,
every entry of $G(x)$ is a polynomial in $x$.
The notation $A\succeq 0$ (resp. $A\succ 0$)
means the matrix $A$ is positive semidefinite (resp. definite).
Suppose $G(x)$ has total degree $2d$ and
\be \label{eq:Gdef}
G(x) =  \sum_{\af \in \N^n: \, \af_1+\cdots+\af_n \leq 2d }
G_\af x_1^{\af_1}\cdots x_n^{\af_n}.
\ee
The $G_\af$ are constant symmetric matrices.
The $G(x) \succeq 0$ is called a polynomial matrix inequality (PMI).
When $G(x)$ is linear, optimizing a linear functional over $S$ becomes a standard
semidefinite programming (SDP) problem. SDP is a very nice convex optimization problem,
has many attractive properties, and can be solved efficiently by numerical methods.
We refer to \cite{NN94,VB,SDPhb}.
%
%
It would be a big advantage if an optimization problem can be formulated in SDP form.
So, we are very interested in knowing when and how
the set $S$ is representable by an SDP.

An elementary approach for this representation problem is to
find symmetric matrices $A_0,A_1,\ldots,A_n$ such that
\be \nn
S=\{x\in \re^n: A_0+A_1 x_1 + \cdots + A_nx_n \succeq 0\}.
\ee
If such $A_i$'s exist, we say $S$
has a linear matrix inequality (LMI) representation and $S$ is LMI representable.
Unfortunately, not every convex set in $\re^n$ is LMI representable.
For instance, the convex set $\{x\in\re^2: 1-x_1^4-x_2^4 \geq 0\}$
is not LMI representable, as proved by Helton and Vinnikov \cite{HV}.
Therefore, we are more interested in finding
a lifted LMI representation, that is, in addition to $A_i$,
finding symmetric matrices $B_1,\ldots,B_N$ such that
\be \label{lftLMI}
S = \left\{ x \in \re^n : \, \exists\, y \in \re^N,
A_0 + \sum_{i=1}^n A_i x_i + \sum_{j=1}^N B_j y_j \succeq 0 \right\}.
\ee
If such matrices $A_i$ and $B_j$ exist, we say $S$ is
semidefinite programming (SDP) representable
or just semidefinite representable,
and \reff{lftLMI} is called a lifted LMI or semidefinite representation for $S$.
The variables $y_j$ are called lifting variables. Nesterov and Nemirovski \cite{NN94},
Ben-Tal and Nemirovski \cite{BTN}, and Nemirovski \cite{N06}
gave collections of convex sets that are SDP representable.
Obviously, to have a lifted LMI, a set must be convex and semialgebraic, i.e.,
it can be defined by a boolean combination of scalar polynomial equalities and inequalities.
However, it is unclear whether every convex semialgebraic set has a lifted LMI or not.

When $G(x)$ is diagonal, i.e., $S$ is defined by scalar polynomial inequalities,
there is some work on the semidefinite representability of $S$.
Parrilo \cite{Par06} constructed lifted LMIs for planar convex sets
whose boundaries are rational planar curves of genus zero.
Lasserre \cite{Las06,Las08} constructed lifted LMIs for convex semialgebraic sets satisfying
certain conditions like bounded degree representation (BDR).
Their constructions use moments and sum of squares techniques.
In \cite{HN1}, Helton and Nie proved sufficient conditions
like sos-convexity and strict convexity,
which justify lifted LMIs from moment type constructions.
Later, in \cite{HN2} they further proved
every compact convex semialgebraic set is SDP representable
if its boundary is nonsingular and positively curved.
Recent work in this area can be found
in \cite{AP09,HN08,Hen08,Lau07,Nie08,NiSt09}.

One might consider to apply the existing results for
the case of scalar polynomial inequalities
like in \cite{HN1,HN2,Las06,Las08,Nie08}
to the case of matrix polynomial inequalities. Note
\be \nn
S = \Big\{x\in \mc{D}:\, p_I(x) \geq 0 \quad \forall\, I \subset \{1,2,\ldots,m\} \Big\}.
\ee
Here $p_I(x)$ are principal minors of $G(x)$ with row (or column) index $I$.
Thus, one could think of studying the semidefinite representability of $S$
by using principal minors $p_I(x)$.
If every $p_I(x)$ is sos-concave or strictly concave over $\mc{D}$, then $S$ is SDP representable
and an explicit lifted LMI would be constructed, as shown in \cite{HN1}.
Unfortunately, this is generally not the case in practice.
The basic reason is that the determinants of $2\times 2$ or bigger matrices
are typically neither concave nor convex, and hence the principle minors $p_I(x)$
would generally be neither concave nor convex.
For instance, when $G(x)$ is linear in $x$,
the minors $p_I(x)$ with $|I|>1$ are typically not concave,
while the set $S$ is clearly LMI and SDP representable.
When $G(x)$ has degree bigger than one,
the minors $p_I(x)$ are also generally not concave, as will be shown by examples later.
Furthermore, $G(x)$ has exponentially many principle minors, and they have much higher degrees.
This is also a big disadvantage for using them in practice.
So, it is usually impractical to study SDP representation
through using principle minors.
Therefore, the conditions directly on $G(x)$ are preferable in applications.
The motivation of this paper is to construct explicit SDP representations for $S$
and prove sufficient conditions directly on $G(x)$ justifying them.

In some applications, $G(x)$ might be given as a matrix rational function, i.e.,
its every entry is rational. This is often the case in control theory.
When $G(x)$ is a scalar rational function,
the author in \cite{Nie08} studied SDP representability of $S$.
In \cite{Nie08}, explicit constructions of lifted LMIs
are given, and sufficient conditions justifying them are proved.
One also might consider to describe $S$ by using polynomials only,
e.g., by multiplying denominators.
However, this kind of processing might destroy matrix concavity,
and usually makes the problem more difficult.
In this paper, we will construct explicit lifted LMIs for $S$ directly based on $G(x)$,
and prove sufficient conditions justifying them.

\bigskip

This paper is organized as follows.
Section~\ref{sec:mat-sos} discusses the semidefinite representation of $S$
when $\mc{D} = \re^n$, and $G(x)$ is polynomial and matrix sos-concave.
Section~\ref{sec:strct-cocav} discusses the semidefinite representation of $S$
when $\mc{D}$ is a compact convex domain,
and $G(x)$ is polynomial and strictly matrix concave on $\mc{D}$.
The case that $G(x)$ is rational and q-module matrix concave over $\mc{D}$
will be discussed in Section~\ref{sec:rat-MI}.

\bigskip
\noindent
{\bf Notations.}\,
The symbol $\N$ (resp., $\re$) denotes the set of nonnegative integers (resp., real numbers).
For any $t\in \re$, $\lceil t\rceil$ denotes the smallest integer not smaller than $t$.
The $\re_+^n$ denotes the nonnegative orthant.
For $x \in \re^n$, $x_i$ denotes the $i$-th component of $x$,
that is, $x=(x_1,\ldots,x_n)$.
When $y$ is a vector indexed by integer vectors in $\N^n$ and $\af \in \N^n$,
$y_\af$ denotes the entry of $y$ whose index is $\af$.
For $\af \in \N^n$, denote $|\af| = \af_1 + \cdots + \af_n$.
For $x \in \re^n$ and $\af \in \N^n$, $x^\af$ denotes $x_1^{\af_1}\cdots x_n^{\af_n}$.
For $\af,\bt \in \N^n$, denote $\af \leq  \bt$ if every $\af_i \leq \bt_i$.
The symbol $\N_{\leq k}$ denotes the multi-index set $\{\af\in\N^n: |\af| \leq k \}$.
For every integer $i\geq 0$, $e_i$ denotes the $i$-th standard unit vector.
The $[x]_d$ denotes the vector of all monomials
having degrees at most $d$ with respect to graded lexicographical ordering, that is,
\[
[x]_d^T = [\, 1 \quad  x_1 \quad \cdots \quad x_n \quad x_1^2 \quad
x_1x_2 \quad \cdots  x_n^2 \quad \cdots
\quad x_1^d \quad x_1^{d-1}x_2 \quad \cdots \quad x_n^d \,].
\]
A polynomial $p(x)$ is said to be a sum of squares (sos)
if there exist finitely many polynomials $q_i(x)$ such that $p(x) = \sum q_i(x)^2$.
A matrix polynomial $H(x)$ is called sos
if there is a matrix polynomial $F(x)$ such that $H(x) = F(x)^TF(x)$.
A polynomial $f(x)$ is called sos-convex if its Hessian is sos,
and  $f(x)$ is sos-concave if $-f(x)$ is sos-convex.
For a set $S$, $int(S)$ denotes its interior, and
$\pt S$ denotes its boundary.
For $u\in \re^N$, $\| u \|_2$ denotes the standard Euclidean norm.
For a matrix $X$, $X^T$ denotes its transpose,
$\|X\|_F$ denotes the Frobenius norm of $X$, i.e., $\|X\|_F = \sqrt{Trace(X^TX)}$,
and $\|X\|_2$ denotes the standard operator $2$-norm of $X$.
The symbol $\bullet$ denotes the standard Frobenius inner product of matrix spaces,
and $I_N$ denotes the $N\times N$ identity matrix.
For a function $f(x)$, $\mc{Z}(f)=\{x\in\re^n: f(x) =0\}$,
$\nabla_{x}f(x)$ denotes its gradient with respect to $x$,
and $\nabla_{xx}f(x)$ denotes its Hessian with respect to $x$.

\section{Matrix sos-concavity} \label{sec:mat-sos}
\setcounter{equation}{0}

This section assumes the domain $\mc{D} = \re^n$ is the whole space
and $G(x)$ is an $m\times m$ symmetric matrix polynomial of degree $2d$.
We will first construct an SDP relaxation for $S$ using moments,
and then prove it is a correct lifted LMI when $G(x)$ satisfies certain conditions.

A natural SDP relaxation of $S$ can be obtained through using moments.
Define linear matrix pencils $\mc{G}(y)$ and $\mc{A}_d(y)$ as
\begin{align*}
\mc{G}(y)  =   \sum_{\af \in\N_{\leq 2d}} y_\af G_\af, \quad
\mc{A}_d(y) =  \sum_{\af \in\N_{\leq 2d}} y_\af A_\af^{(d)},
\end{align*}
where $G_\af$ are from \reff{eq:Gdef} and $A_\af^{(d)}$ are such that
\be \label{def:A_af}
[x]_d [x]_d^T = \sum_{\af \in\N_{\leq 2d}} x^\af A_\af^{(d)}.
\ee
Since
$
S = \left\{ x\in \re^n:\, G(x) \succeq 0, \,
[x]_d[x]_d^T \succeq 0 \right\},
$
we know
\[
S = \left\{ (y_{e_1},\ldots,y_{e_n}) \in \re^n:\,
\exists \, x \in \re^n, \, y_\af = x^\af \,\, \forall \, \af \in \N_{\leq 2d}, \,
\mc{G}(y) \succeq 0, \, \mc{A}_d(y) \succeq 0
\right\}.
\]
Here, each $e_i$ denotes the $i$-th standard unit vector
whose only nonzero entry is one at index $i$.
If the condition $y_\af = x^\af$ is removed in the above,  then $S$ is a subset of
\be \label{def:L-sos}
L = \left\{ x \in \re^n
\left| \baray{c}
\exists \, y \in \re^{\binom{n+2d}{2d}}, \,
\mc{G}(y)  \succeq 0, \, \mc{A}_d(y) \succeq 0, \\
y_0 = 1, x_1 = y_{e_1}, \ldots, x_n = y_{e_n}
\earay \right.
\right\}.
\ee
So, $S\subseteq L$. Does $S=L$? What conditions make $S=L$?
We look for sufficient conditions guaranteeing $S=L$.

The matrix-valued function $G(x)$ is called {\it matrix concave} over
a convex domain $\mc{D}$ if for all $u,v \in \mc{D} $ and $0\leq \theta \leq 1$
it holds that
\[
G(\theta u + (1-\theta)v ) \succeq  \theta G(u) + (1-\theta)G(v).
\]
In the above, when $\mc{D}=\re^n$, we just say
$G(x)$ is matrix concave.
The matrix concavity of $G(x)$ over $\mc{D}$ is equivalent to
\[
-\nabla_{xx} (\xi^T G (x) \xi) \succeq 0
\quad \forall \,  \xi\in \re^m, \, \forall \, x \in \mc{D}.
\]
We would like to point out that $G(x)$ might not be matrix concave while $S$ is still convex.
For instance, the quadratic polynomial matrix inequality
\[
Q(x) := \bbm
x_1x_2 +2 & x_1x_2 & 0 \\
x_1x_2 & x_1x_2 - 1 & 0 \\
0 & 0 & x_1+x_2
\ebm \succeq 0
\]
defines the convex set
$\{x\in \re_+^2: x_1x_2 \geq 2 \}$, but $Q(x)$ is not matrix concave on $\re_+^2$.

Generally, it is difficult to check matrix concavity.
Even for the simple case of quadratic matrix polynomials,
the problem is already NP-hard, as shown below.

\begin{pro}
It is NP-hard to check the matrix concavity of quadratic matrix polynomials.
\end{pro}
\begin{proof}
Let $\overline{m} = \half m(m+1)$.
For any symmetric matrices $A_1,\ldots, A_{\overline{m}} \in \re^{m\times m}$ and
$B_1,\ldots, B_{\overline{m}} \in \re^{n\times n}$, define the matrix polynomial
\[
G(x) = -\half \sum_{i=1}^{\overline{m}}  (x^TB_ix)A_i.
\]
Then we have
\[
-\nabla_{xx} (\xi^T G (x) \xi) = \sum_{i=1}^{\overline{m}}  (\xi^TA_i\xi) B_i.
\]
So, $G(x)$ is matrix concave if and only if the following bi-quadratic form in $(\xi,z)$
\[
\sum_{i=1}^{\overline{m}} (\xi^TA_i\xi) (z^TB_iz)
\]
is nonnegative everywhere. It has been proven in \cite{LNQY} that
it is NP-hard to check the nonnegativity of bi-quadratic forms.
Therefore, it must also be NP-hard to check the matrix concavity of quadratic $G(x)$.
\end{proof}

A stronger but easier checkable condition than matrix concavity
is the so called matrix sos-concavity.
We say $G(x)$ is {\it matrix sos-concave} if
for every $\xi \in \re^m$ there exists a matrix polynomial $F_{\xi}(x)$ in $x$ such that
\be  \label{MATsoscav}
- \nabla_{xx} (\xi^T G(x) \xi) = F_{\xi}(x)^T F_{\xi}(x).
\ee
The above $F_{\xi}(x)$ has $n$ columns but its number of rows
might be different from $n$, and its coefficients of $x^\af$ depend on $\xi$.
Note that when $G(x)$ is quadratic, $G(x)$ is matrix concave if and only if
it is matrix sos-concave. This is because $- \nabla_{xx} (\xi^T G(x) \xi)$
is independent of $x$, and for fixed $\xi$ it is positive semidefinite
if and only if it is sos (Cholesky factorization).

\begin{theorem} \label{thm:mat-sos}
Suppose $G(\tilde{x})\succ 0$ for some $\tilde{x}$.
If $G(x)$ is matrix sos-concave, then $S=L$.
\end{theorem}
\begin{proof}
We have already seen $S \subseteq L$, so it suffices to prove the reverse containment.
Suppose otherwise $L\ne S$, then there must exist a point $\hat x \in L/S$.
Since $S$ is closed and convex, by Hahn-Banach Theorem,
there exists a supporting hyperplane
$
\mc{H}=\{x\in\re^n:\, a^Tx \geq b \} \supseteq S
$
such that $a^T u = b$ for some $u\in S$ and $a^T \hat x < b$.
Consider the linear optimization problem
\be \label{linopt:G>=0}
\min_{x\in\re^n} \, a^Tx  \quad \, \mbox{subject to} \quad \, G(x) \succeq 0.
\ee
Clearly $u$ is a minimizer and $b$ is the optimal value.
The optimization problem \reff{linopt:G>=0} is convex.
The existence of $\tilde{x}$ with $G(\tilde{x})\succ 0$,
i.e., the Slater's condition holds, implies
there exists a matrix Lagrange multiplier $\Lambda \succeq 0$ such that
\[
\Lambda \bullet G(u) =0, \quad
a = \nabla_x (\Lambda \bullet G(x))\Big|_{x=u}.
\]
The value and gradient of $a^Tx - \Lambda \bullet G(x) - b$ vanish at $u$.
Then, by the Taylor expansion at $u$, we have
\begin{align*}
a^Tx - \Lambda \bullet G(x) - b  =
(x-u)^T \left(\int_0^1\int_0^t
- \nabla_{xx} (\Lambda \bullet G(u+s(x-u))\,ds\,dt  \right) (x-u).
\end{align*}
Since $\Lambda\succeq 0$, there exist vectors $\lmd^{(k)}$ such that
$ \Lambda = \sum_{k=1}^K \lmd^{(k)} (\lmd^{(k)})^T$. So, we have
\begin{align*}
 & \quad a^Tx - \Lambda\bullet G(x) - b  = \\
\sum_{k=1}^K (x-u)^T  & \left(\int_0^1\int_0^t -
\nabla_{xx} ( (\lmd^{(k)})^T G(u+s(x-u)) \lmd^{(k)} \,ds\,dt  \right) (x-u).
\end{align*}
Since $G(x)$ is matrix sos-concave, by Lemma~7 in \cite{HN1},
we know each summand in the above must be sos.
Thus $a^Tx - \Lambda \bullet G(x) - b$ must also be an sos polynomial of degree $2d$.
So, there exists a symmetric matrix $W\succeq 0$ such that the identity
\[
a^Tx - \Lambda \bullet G(x) - b =  [x]_d^T W [x]_d =
W \bullet \left( [x]_d[x]_d^T\right)
\]
holds. By definition of matrices $A_\af^{(d)}$ in  \reff{def:A_af}, we have
\[
a^Tx - \Lambda \bullet G(x) - b =
W \bullet \left(\sum_{\af\in\N_{\leq 2d}} x^\af A_\af^{(d)} \right).
\]
Since $\hat{x} \in L$, there exists $\hat{y}$ such that
$\hat{x}=(\hat{y}_{e_1},\ldots,\hat{y}_{e_n})$,
$\mc{G}(\hat{y})\succeq 0$ and $\mc{A}_d(\hat{y})\succeq 0$.
So, if we replace $ \hat{x}^\af$ by $\hat{y}_\af$ in the above identity,
then
\[
a^T\hat{x} - \Lambda \bullet \mc{G}(\hat{y}) - b =
W \bullet \left(\sum_{ |\af| \leq 2d } \hat{y}_\af A_\af^{(d)} \right)
= W \bullet \mc{A}_d(\hat{y}),
\]
or equivalently
\[
a^T\hat{x} - b =  \Lambda \bullet \mc{G}(\hat{y})  +  W \bullet \mc{A}_d(\hat{y}).
\]
Since $\Lambda,\mc{G}(\hat{y}),W,\mc{A}_d(\hat{y})\succeq 0$,
we must have $a^T\hat{x} - b \geq 0$,
which contradicts the previous assertion that $a^T\hat x - b <0$. So, $S=L$.
\end{proof}

\begin{exm} \label{emp:nonunf-matsos}
Consider the set $S=\left\{x\in \re^3: G(x)\succeq 0\right\}$ where
\[
G(x) =
\bbm
2- x_1^2-2x_3^2 & 1+x_1x_2 & x_1x_3  \\
1+x_1x_2 & 2-x_2^2-2x_1^2 & 1+x_2x_3 \\
x_1x_3 & 1+x_2x_3 & 2-x_3^2-2x_2^2
\ebm.
\]
\begin{figure}
\centering
\includegraphics[width=.66\textwidth]{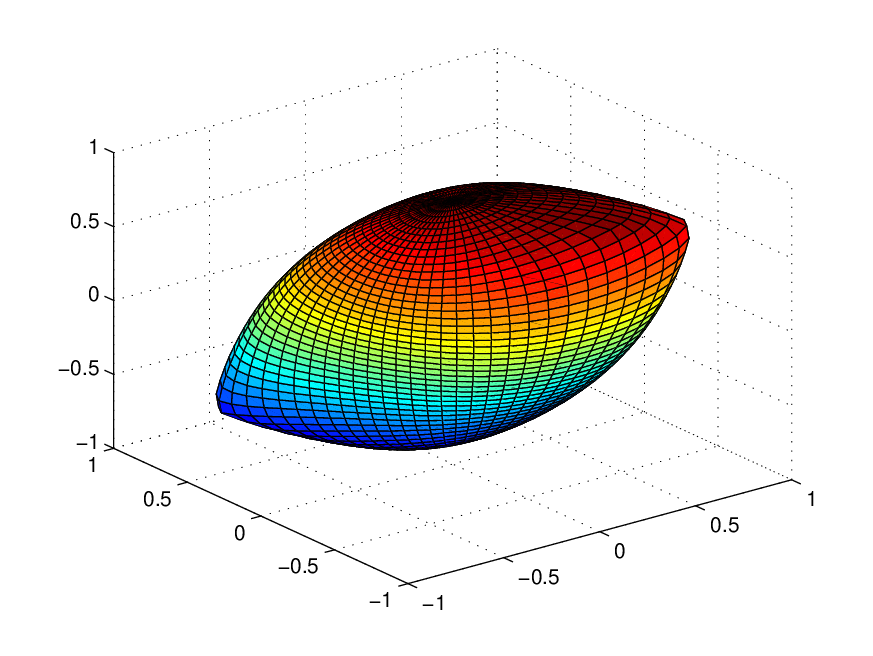}
\caption{ The drawn body is the convex set in Example~\ref{emp:nonunf-matsos}.}
\label{fig:nonunf-matsos:var3deg2}
\end{figure}
The Hessian $-\nabla_{xx}(\xi^TG(x)\xi)$ is positive semidefinite for all $\xi \in \re^3$.
This is because
\[
-\half \nabla_{xx}(\xi^TG(x)\xi) \, =  \,
\bbm
\xi_1^2+2\xi_2^2 & -\xi_1\xi_2 & -\xi_1\xi_3  \\
-\xi_1\xi_2 & \xi_2^2+2\xi_3^2 & -\xi_2\xi_3 \\
-\xi_1\xi_3 & -\xi_2\xi_3 & \xi_3^2+2\xi_1^2
\ebm \succeq 0 \quad \forall \, \xi\in \re^3,
\]
which is due to the fact that the bi-quadratic form
$z^T\big(-\half \nabla_{xx}(\xi^TG(x)\xi)\big)z$ in $(z,\xi)$
\[
z_1^2\xi_1^2+z_2^2\xi_2^2+z_3^2\xi_3^2+
2(z_1^2\xi_2^2+z_2^2\xi_3^2+z_3^2\xi_1^2)
-2(z_1z_2\xi_1\xi_2+z_1z_3\xi_1\xi_3+z_2z_3\xi_2\xi_3)
\]
is nonnegative everywhere, as shown by Choi \cite{Choi}.
So, this $G(x)$ is matrix sos-concave, because for every fixed $\xi$
the Hessian $-\nabla_{xx}(\xi^TG(x)\xi)$ is a constant matrix
which is positive semidefinite and must be sos.
Thus, we know $S$ is convex and by Theorem~\ref{thm:mat-sos} a lifted LMI for it is
\be \label{lmi:emp2.3}
\left\{
\baray{l}
x \in \re^3: \, \exists \, y_\af \, (\af\in\N_{\leq 2}) \quad
\mbox{ such that }  \\
\bbm
2- y_{200}-2y_{002} \!\!\!\! &  1+y_{110}  & y_{101}  \\
1+y_{110} & \!\!\!\!\!\!\!\!\! 2- y_{020} -2y_{200} \!\!\!\!\!\!\!\!\! &  1+y_{011}  \\
y_{101} & 1+y_{011} &  \!\!\! 2- y_{002} -2 y_{020}
\ebm \succeq 0,
\earay
\bbm
 1 & x_1 & x_2 & x_3 \\
x_1 & y_{200} &  y_{110} & y_{101}  \\
x_2 & y_{110} &  y_{020} &  y_{011}  \\
x_3 & y_{101} &  y_{011} & y_{002}
\ebm \succeq 0
\right\}.
\ee
A picture of the set $S$ is in Figure~\ref{fig:nonunf-matsos:var3deg2}.
It would be drawn by finding its boundary points
in various directions sampled on the unit sphere,
e.g., by making a fine enough grid.
\qed
\end{exm}

The matrix sos-concavity condition requires checking the Hessian
\[
- \nabla_{xx} (\xi^T G(x) \xi)
\]
is sos for every $\xi\in \re^m$.
This is almost impossible in applications.
However, a stronger condition called {\it uniformly matrix sos-concave} is
\[
-\nabla_{xx} (\xi^T G(x) \xi)  = F(\xi,x)^T F(\xi,x),
\]
where $F(\xi,x)$ is now a matrix polynomial in joint variables $(\xi,x)$.
It is easier to check. The uniform matrix sos-concavity can be
verified by solving a single SDP feasibility problem
(see Section~3 of \cite{HN1}).
Clearly, the following is a consequence of Theorem \ref{thm:mat-sos}.

\begin{cor} \label{cor:unf_matsos}
Suppose $G(\tilde{x})\succ 0$ for some $\tilde{x}$ .
If $G(x)$ is uniformly matrix sos-concave, then $S=L$.
\end{cor}

It should be pointed out that when $G(x)$ is matrix sos-concave,
it is not necessarily that $G(x)$ is uniformly matrix sos-concave.
For a counterexample, consider the $G(x)$ of Example~\ref{emp:nonunf-matsos}.
For any fixed $\xi\in \re^m$, the Hessian $-\nabla_{xx}(\xi^TG(x)\xi)$ there is independent of $x$
(since $G(x)$ is quadratic),
and it is sos if and only if it is positive semidefinite.
But if we think of $\xi$ as an indeterminant vector,
then $-\nabla_{xx}(\xi^TG(x)\xi)$ is not sos in $\xi$, as shown by Choi~\cite{Choi}.
Now let us see an example of uniformly matrix sos-concave $G(x)$.

%

\begin{exm}  \label{emp:unf-matsos}
Consider the set $S=\left\{x\in \re^2: G(x)\succeq 0\right\}$ where
\[
G(x) =
\bbm 2-2x_1^4-4x_1^2x_2^2-2x_2^4  & 3-x_1^3x_2-x_1x_2^3  \\
3-x_1^3x_2-x_1x_2^3  & 5-x_1^4-4x_1^2x_2^2-x_2^4  \ebm.
\]
\begin{figure}
\centering
\includegraphics[width=.66\textwidth]{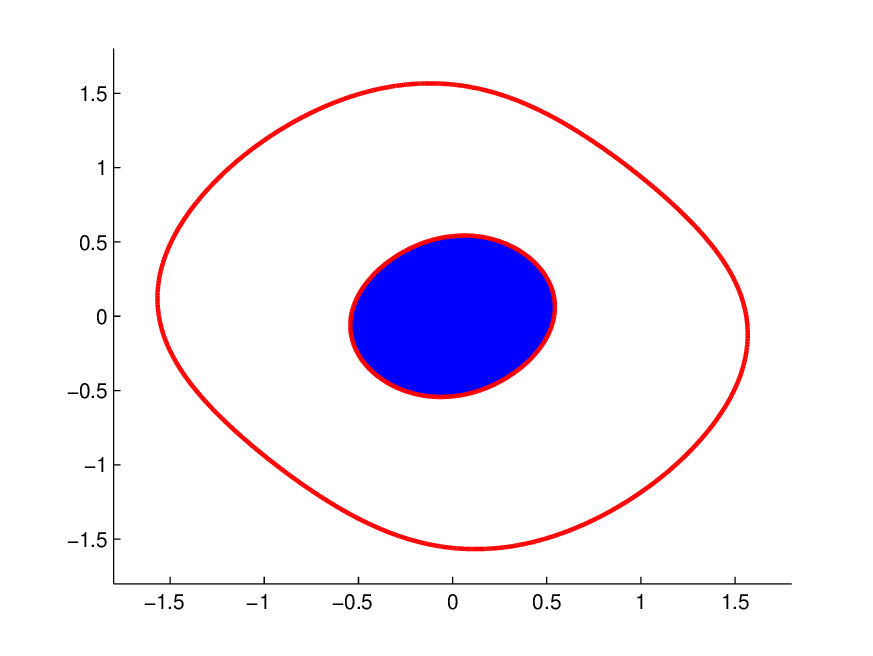}
\caption{ The shaded area is the convex set in Example~\ref{emp:unf-matsos},
and the curve is $\det G(x)=0$.}
\label{fig:unf-matsos:var2deg4}
\end{figure}
The above $G(x)$ is uniformly matrix sos-concave because
\begin{align*}
-\nabla_{xx} (\xi^TG(x)\xi) = H_1+H_2+H_3+H_4,
\end{align*}
\[
H_1 = 2 \bbm  2 \xi_1 x_1 +\xi_2 x_2 \\ 2 \xi_1 x_2 +\xi_2 x_1 \ebm
\bbm  2 \xi_1 x_1 +\xi_2 x_2 \\ 2 \xi_1 x_2 +\xi_2 x_1 \ebm^T, \quad
H_2 = 8(\xi_1^2+\xi_2^2) \bbm x_1^2  & x_1x_2 \\ x_1x_2 & x_2^2 \ebm,
\]
\[
H_3 = 2
\bbm \xi_1 x_1 & \xi_2 x_2 & \xi_2 x_1 \\
\xi_2 x_1 & \xi_1 x_2 & \xi_2 x_2 \ebm
\bbm \xi_1 x_1 & \xi_2 x_2 & \xi_2 x_1 \\
\xi_2 x_1 & \xi_1 x_2 & \xi_2 x_2 \ebm^T,
\]
\[
H_4 = 2 \left( \left((\xi^Tx)^2+\xi_2^2x_1^2\right)\bbm 1 & 0 \\ 0 & 1 \ebm
+ \xi_1^2 \bbm  2x_1^2+4x_2^2 & 0 \\ 0 & 3x_1^2+3x_2^2 \ebm \right).
\]
So, this set $S$ is convex, and by Corollary \ref{cor:unf_matsos} a lifted LMI for it is
\[
\left\{
\baray{l}
x \in \re^2: \,
\exists \, y_{ij} \, (0\leq i, j \leq 4) \, \mbox{ such that } \\ \qquad  \\
\bbm 2-2(y_{40}+2y_{22}+y_{04}) \!\!\!\!\!\!\! & 3-(y_{31}+y_{13}) \\
3-(y_{31}+y_{13}) & \!\!\!\!\!\!\!  5 - (y_{40}+3y_{22}+y_{04}) \ebm \succeq 0,
\earay
\bbm
1 & x_1 & x_2 & y_{20} & y_{11}  & y_{02} \\
x_1 & y_{20} & y_{11} & y_{30} & y_{21}  & y_{12} \\
x_2 & y_{11} & y_{02} & y_{21} & y_{12}  & y_{03} \\
y_{20} & y_{30} & y_{21} & y_{40} & y_{31}  & y_{22} \\
y_{11} & y_{21} & y_{12} & y_{31} & y_{22}  & y_{13} \\
y_{02} & y_{12} & y_{03} & y_{22} & y_{13}  & y_{04}
\ebm \succeq 0
\right\}.
\]
Let $F$ be the above LMI. The region of $x$ satisfying $F$
would be plotted by function
{\tt plot} provided in software {\tt YALMIP} \cite{yalmip},
which is drawn in the shaded area of
Figure~\ref{fig:unf-matsos:var2deg4}.
Clearly, the boundary of $S$ lies on the curve $\det G(x) = 0$,
which is also drawn in Figure~\ref{fig:unf-matsos:var2deg4}.
It has two connected components.
The inner one surrounds $S$ and is its boundary $\pt S$.
The outer one does not touch $S$, because $G(x)\succeq 0$ fails there.
So, Figure~\ref{fig:unf-matsos:var2deg4} confirms $F$ in the above represents $S$.
In this example, the determinant $\det G(x)$ is
neither concave nor convex.
\qed
\end{exm}

In the following, we list some classes of $G(x)$
that is (uniformly) matrix sos-concave.

\bnum

\item Suppose $G(x)$ is of the form
\[
G(x) \, = \, A_0(x)+f_1(x)A_1+\cdots+f_k(x)A_k
\]
where $A_0(x)$ is linear in $x$, every $f_i(x)$ is an sos-concave scalar polynomial,
and $A_1,\ldots,A_k \succeq 0$. Then $G(x)$ is uniformly matrix sos-concave because
\[
-\nabla_{xx} \, (\xi^TG(x)\xi) \, = \,
(-\nabla_{xx} f_1(x) ) (\xi^TA_1\xi) + \cdots + (-\nabla_{xx} f_k(x) ) (\xi^TA_k\xi)
\]
is sos in $(x,\xi)$. Such an example is
\[
\bbm 1 & 0 \\ 0 & 1 \ebm - x_1^4 \bbm 1 & 1 \\ 1 & 1 \ebm
- x_2^4 \bbm 1 & -1 \\ -1 & 1 \ebm.
\]
The determinant of the above is not concave in $\re^2$
(also not sos-concave).

\item Suppose $G(x)$ is of the form
\[
G(x) \, = \, F(x)+\diag(f_1(x),\ldots, f_m(x))
\]
where $F(x)$ is linear in $x$ and each $f_i(x)$ is a scalar polynomial.
For every $\xi$, we have
\[
\nabla_{xx} \, (\xi^TG(x)\xi) \, = \, \sum_{i=1}^m \, \xi_i^2  \nabla_{xx} f_i(x) .
\]
Clearly, $G(x)$ is matrix concave if and only if
every $f_i(x)$ is concave,
and $G(x)$ is matrix sos-concave if and only if
every $f_i(x)$ is sos-concave,
which is also equivalent to that
$G(x)$ is uniformly matrix sos-concave. Such an example is
\[
\bbm x_1 & x_2 \\ x_2 & x_1 \ebm - \bbm x_1^4 & 0 \\ 0 & x_2^4 \ebm.
\]
Its determinant is not concave in $\re^2$ (also not sos-concave).

\item  Suppose $G(x)$ is of the form
\[
G(x) \, = \, A(x) - Q(x)
\]
where $A(x)$ is linear in $x$ and $Q(x)$ is quadratic
and positive semidefinite everywhere. This $G(x)$ must be matrix sos-concave.
For every $\xi$, we have
\[
-\nabla_{xx}~\xi^TG(x)\xi \, = \, \nabla_{xx}~\xi^TQ(x)\xi.
\]
Since $Q(x)\succeq 0$ for all $x$, the quadratic polynomial $\xi^TQ(x)\xi$
is nonnegative everywhere, and there exists a symmetric matrix $W=W(\xi)\succeq 0$ such that
\[
\xi^TQ(x)\xi = x^T W x, \qquad \nabla_{xx}~\xi^TQ(x)\xi = 2W.
\]
Thus, the $G(x)$ is matrix sos-concave,
and $L$ in \reff{def:L-sos} is a lifted LMI for the set defined by
\[
A(x) - Q(x)\succeq 0.
\]
This generalizes the following result:
if $q(x)$ is a nonnegative quadratic scalar polynomial, then
for any linear $a(x)$ the set defined by
\[
a(x) - q(x) \geq 0
\]
is convex and SDP representable.
Such a $G(x)$ is given in Example~\ref{emp:nonunf-matsos}.

\item  Suppose $n=2$ and $G(x)$ is of the form
\[
G(x) \, = Q_1(x) - Q_2(x)-Q_{2d}(x)
\]
where $Q_1(x)$ is linear in $x$, $Q_2(x)$ is quadratic in $x$,
and $Q_{2d}(x)$ is homogeneous of degree $2d$.
Then, for any given $\xi\in\re^m$ the Hessian
\begin{align*}
-\nabla_{xx} \, (\xi^TG(x)\xi) \, =  \, \nabla_{xx} \, (\xi^TQ_2(x)\xi)
+ \nabla_{xx} \, (\xi^TQ_{2d}(x)\xi) \succeq 0 \quad \forall x\in\re^n
\end{align*}
if and only if both $\nabla_{xx} \, (\xi^TQ_2(x)\xi)$ and
$\nabla_{xx} \, (\xi^TQ_{2d}(x)\xi)$ are positive semidefinite for every $x$.
The Hessian $\nabla_{xx}(\xi^TQ_2(x)\xi)$ is independent of $x$.
Note that every bivariate homogeneous positive semidefinite matrix polynomial is sos
(see \cite[Theorem~7.1]{CLR}).
In this case, $G(x)$ is matrix concave if and only if it is matrix sos-concave.

\item Suppose $n=1$, then $G(x)$ is matrix concave if and only if
\begin{align*}
P(x):=\Big(- G_{ij}^{\prm\prm}(x)  \Big)_{1\leq i,j\leq m} \succeq 0 \quad \forall \, x\in \re,
\end{align*}
which is equivalent to that $P(x)$ is sos.
This is because every univariate
positive semidefinite matrix polynomial is sos \cite[Theorem~7.1]{CLR}.
In this case, the matrix concavity coincides with uniform matrix sos-concavity,
and $G(x)\succeq 0$ defines an interval like $[a,b]$.
Typically, the end points $a,b$ are algebraic (but not rational) functions of the coefficients of $G$.
However, the parameters of $L$ are rational in the coefficients of $G$.
This is interesting when a rational SDP representation is preferable.

\enum

\section{Strict matrix concavity} \label{sec:strct-cocav}
\setcounter{equation}{0}

This section assumes $S=\{x\in\mc{D}: G(x) \succeq 0\}$ and
\[
\mc{D}=\{x\in\re^n: g_1(x) \geq 0, \ldots, g_r(x) \geq 0  \}
\]
is a domain defined by polynomials $g_1, \ldots, g_r$.
When $\mc{D}$ is compact convex and
$G(x)$ is strictly matrix concave on $\mc{D}$, we will show that
$S$ is semidefinite representable, and a lifted LMI for it is
explicitly constructible.

Like in the previous section, a natural SDP relaxation of $S$ is constructible
by using moments. Let $g_0\equiv 1$ and
\[
d = \max\left\{ \deg(G(x))/2, \lceil\deg(g_k)/2\rceil, \, k=0,1,\ldots,r \right\}.
\]
For every integer $N\geq d$ and $k=0,\ldots,r$,
define symmetric matrices $B_{k,\bt}^{(N)}$ such that
\be  \label{def:B_af}
g_k(x) [x]_{N-d_k}  [x]_{N-d_k}^T =
\sum_{\bt\in\N_{\leq 2N}} x^\bt B_{k,\bt}^{(N)}, \quad
d_k =  \lceil\deg(g_k)/2\rceil.
\ee
This determines $B_{k,\bt}^{(N)}$ uniquely.
Then, define the linear matrix pencils $B_k^{(N)}(y)$ as
\begin{align*}
B_k^{(N)}(y) & =  \sum_{\bt\in\N_{\leq 2N}} y_\bt B_{k,\bt}^{(N)}, \quad
k=0,1,\ldots,r.
\end{align*}
Clearly, $S$ can be equivalently described as
\[
S=\left\{(y_{e_1},\ldots,y_{e_n})
\left|\baray{l}
\exists x\in \re^n,  y_\af = x^\af \quad \forall \, \af \in \N_{\leq 2N}, \\
\mc{G}(y)  \succeq 0, \, B_k^{(N)}(y) \succeq 0, k=0,\ldots,r
\earay\right.
\right\}.
\]
If the conditions $y_\af = x^\af$ are removed in the above, then $S$ is contained in the set
\be \label{def:L_N}
L_N = \left\{ x \in \re^n
\left|\baray{l}
\exists  y \in \re^{\binom{n+2N}{n}},
y_0 = 1, x_1 = y_{e_1}, \ldots, x_n = y_{e_n}, \\
\mc{G}(y)  \succeq 0, \, B_k^{(N)}(y) \succeq 0, k=0,\ldots,r
\earay \right.
\right\}.
\ee
So, we have $S \subseteq L_N$ for every $N\geq d$.
It is clear that $L_{N+1} \subseteq L_N$, because $L_{N+1}$ is a restriction of $L_N$.
Thus, it holds the nesting containment relation:
\[
L_{d} \supseteq \cdots \supseteq L_{N} \supseteq L_{N+1} \supseteq
\cdots \supseteq S.
\]
Does there exist a finite $N$ such that $L_N=S$? What conditions on $S$
make it true? In the following, we look for sufficient conditions
guaranteeing $L_N=S$.

Semidefinite representation of $S$ is closely related to
linear functionals nonnegative on $S$.
For a given $ 0 \ne \ell \in \re^n$, consider the linear optimization problem
\be \label{eq:linopt}
\underset{x\in int(\mc{D})}{\min} \quad \ell^Tx \quad
\mbox{subject to} \quad G(x) \succeq 0.
\ee
When $S\subseteq int(\mc{D})$ is compact, \reff{eq:linopt}
always has a minimizer $u\in \pt S \cap int(\mc{D})$.
If further there exists $\tilde{x}\in \mc{D}$ such that $G(\tilde{x})\succ 0$
({\it Slater's condition} holds) and $G(x)$ is matrix concave in $\mc{D}$,
then there exists $0 \preceq \Lambda \in \re^{m\times m}$ such that
(see \cite[p. 306]{Shp})
\be \label{lag:Lambda}
\Lambda \bullet G(u) =0, \quad
\ell = \nabla_x (\Lambda \bullet G(x)) \big|_{x=u}.
\ee
Thus, by its Taylor expansion at $u$, we know $\ell^T(x-u) - \Lambda \bullet G(x)$ equals
\be \label{Tay:strct-cav}
\baray{c}  (x-u)^T \cdot
\underbrace{\left(\int_0^1\int_0^t -\nabla_{xx}
(\Lambda \bullet G(u+s(x-u)) \, )\,ds\,dt\right)}_{H(u,x)}
\cdot (x-u).
\earay
\ee
If the above matrix polynomial $H(u,x)$
has a weighted sos representation in terms of $G(x)$ and $g_i(x)$,
then we can also get a similar one for $\ell^T(x-u)$.
For this purpose, we need some assumptions on $\mc{D}$ and $G(x)$.



\begin{ass} \label{ass:actDefHen}
(i) $G(x)$ is matrix concave on $\mc{D}$, and $G(x)$ satisfies
\be \nn
- \nabla_{xx} ( \Lambda \bullet G(u))
\succ 0 \quad \forall\, u\in \pt S, \quad \forall \, 0\ne \Lambda \succeq 0.
\ee
(ii) The archimedean condition (AC) holds for $\mc{D}$, i.e.,
there exist $M>0$ and sos polynomials $s_0, s_1, \ldots, s_r$ satisfying
\[
M - \|x\|_2^2 = s_0 + g_1s_1 + \cdots + g_rs_r.
\]
\end{ass}

Note that the item (i) of Assumption~\ref{ass:actDefHen}
is equivalent to $- \nabla_{xx} ( \xi^T G(u)\xi ) \succ 0 $
for every $u \in \pt S$ and every $0\ne \xi \in \re^m$.
Clearly, AC in Assumption~\ref{ass:actDefHen} implies
the domain $\mc{D}$ is compact, since $M-\|x\|_2^2\geq 0$
for all $x\in \mc{D}$.

\begin{theorem}  \label{thm:defHen}
Suppose $S \subseteq int(\mc{D})$, $\mc{D}$ is convex, and
$G(\tilde{x})\succ 0$ for some $\tilde{x}\in \mc{D}$.
If Assumption~\ref{ass:actDefHen} holds, then $S=L_N$ for all $N$ big enough.
\end{theorem}
\begin{proof}
For a matrix polynomial $G(x)$ given in \reff{eq:Gdef},
we define its norm $\|G(x)\|$ as
\[
\|G(x)\| = \max_{\af\in\N_{\leq 2d}} \,
\frac{\af_1 ! \cdots \af_n !}{ |\af| !} \, \|G_\af\|_2.
\]
The AC of Assumption~\ref{ass:actDefHen} implies $\mc{D}$ is compact.
So, there exists $\Dt>0$ such that
\[
\|\Lambda\|_F^{-1} \cdot \| H(u,x) \|
\leq \Dt \quad
\forall \, 0 \ne \Lambda \in \re^{n\times n}, \quad  \forall \,
u\in \pt S, \quad \forall \,x\in \mc{D}.
\]
Here $H(u,x)$ is defined in \reff{Tay:strct-cav}.
Assumption~\ref{ass:actDefHen} implies $H(u,x)\succ 0$ for all $u\in \pt S$ and $x\in \mc{D}$.
This is because otherwise if $H(u,x)$ is not positive definite,
we can find $0\ne v \in \re^n$ such that $v^TH(u,x)v=0$, i.e.,
\[
\int_0^1\int_0^t v^T \Big( -\nabla_{xx}
\big(\Lambda \bullet G(u+s(x-u)\big) \, \Big)v\,ds\,dt =0.
\]
Since $G(x)$ is matrix concave on the convex domain $\mc{D}$, we must have
\[
v^T \Big( -\nabla_{xx} \big(\Lambda \bullet G(u+s(x-u)\big) \Big)v = 0 \quad
\forall \, 0\leq s \leq t\leq 1.
\]
In particular, we get $v^T \Big( -\nabla_{xx} \big(\Lambda \bullet G(u)\big) \Big)v = 0$,
which contradicts Assumption~\ref{ass:actDefHen}.
Therefore, by the compactness of $\pt S$ and $\mc{D}$, there exists $\dt>0$ satisfying
\[
\|\Lambda\|_F^{-1} \cdot H(u,x)
\succeq \dt I_n \quad \forall \, u\in \pt S,
\quad \forall \, x \in \mc{D}, \quad \forall \, 0\ne \Lambda \succeq 0.
\]
By Theorem~29 in \cite{HN1} and the AC for $\mc{D}$,
there exists an integer $N^*$ such that
for every $0\ne \Lambda \succeq 0$ and  $u\in \pt S$,
there exist sos matrices $F_0(x), F_1(x),\ldots, F_r(x)$ satisfying
\be  \label{eq:ballSOS}
\|\Lambda\|_F^{-1} \cdot H(u,x) = \sum_{k=0}^r g_k(x) F_k(x),
\ee
\[
\deg(F_k)+ 2 d_k \leq 2 (N^*-1), \, k=0,,\ldots, r.
\]

Now we claim $S=L_{N^*}$. Since $S\subseteq L_{N^*}$, we need to show
$L_{N^*}\subseteq S$. Suppose otherwise there exists $\hat{x} \in L_{N^*}/S$.
Since $\mc{D}$ is compact convex and $G(x)$ is matrix concave on $\mc{D}$,
$S$ is closed and convex. By Hahn-Banach Theorem,
there exist $0\ne \ell \in \re^n$ and $u\in \pt S$ satisfying
\[
\ell^T(x-u) \geq 0 \quad \forall \, x\in S, \quad \ell^T(\hat{x}-u) <0.
\]
Consider the linear optimization problem \reff{eq:linopt} with this $\ell$.
The point $u \in \pt S$ is a minimizer of \reff{eq:linopt},
and it is also a local minimizer of
\be \label{linopt:uncon}
\underset{x\in \re^n}{\min} \quad \ell^Tx \quad
\mbox{subject to} \quad G(x) \succeq 0.
\ee
Since $G(\tilde{x})\succ 0$ and $G(x)$ is matrix concave, it holds that
\[
G(u) + \sum_{i=1}^n (\tilde{x}_i-u_i) \frac{\pt G(x)}{\pt x_i}\big|_{x=u}
\succeq G(\tilde{x}) \succ 0.
\]
This means the Mangasarian-Fromovitz (MF) condition holds
at $u$ for optimization problem \reff{linopt:uncon},
and thus the first order necessary condition holds at $u$
(see \cite[p.~306]{Shp}).
So, there exists $\Lambda \succeq 0$ satisfying \reff{lag:Lambda}.
From \reff{Tay:strct-cav} and \reff{eq:ballSOS}, we know there exist sos polynomials
$p_0$, $p_1$, \ldots, $p_r$ satisfying
\begin{align*}
\ell^T(x -u) = \Lambda \bullet G(x) + \sum_{k=0}^r p_k(x) g_k(x), \\
\deg(p_k)+2d_k \leq 2N^*, \quad k=0,1,\ldots, r.
\end{align*}
So, there are symmetric matrices $W_0, W_1, \ldots, W_r\succeq 0$ such that
\[
\ell^T(x - u) =  \Lambda \bullet G(x) +
\sum_{k=0}^r  g_k(x)  [x]_{N^*-d_k}^T W_k [x]_{N^*-d_k}.
\]
By definition of matrices $B_{k,\bt}^{(N^*)}$ in \reff{def:B_af}, it holds the identity
\[
\ell^T (x-u) =  \Lambda \bullet G(x) +
\sum_{k=0}^r W_k \bullet \left( \sum_{\bt\in\N_{\leq 2d}} x^\bt B_{k,\bt}^{(N^*)} \right).
\]
By the choice of $\hat{x}$, there exists $\hat{y}$ such that
$\hat{x} = (\hat{y}_{e_1},\ldots, \hat{y}_{e_n})$,
$\mc{G}(\hat{y})\succeq 0$, and every $B_k^{N^*}(\hat{y})\succeq 0$.
So, if each $ \hat{x}^\af$ is replaced by $\hat{y}_\af$ in the above, then
\[
\ell^T ( \hat{x} - u) =  \Lambda \bullet \mc{G}(\hat{y}) +
\sum_{k=0}^r W_k \bullet B_k^{N^*}(\hat{y}) \geq 0,
\]
which contradicts $\ell^T(\hat x - u) <0$. Hence, we must have $S=L_{N^*}$.

For every $N\geq N^*$, the relation $S\subseteq L_{N} \subseteq L_{N^*}$ implies
$S=L_{N}$.
\end{proof}

Assumption~\ref{ass:actDefHen} requires to check
$-\nabla_{xx}(\Lambda \bullet G(u))\succ 0$
for every nonzero $\Lambda \succeq 0$ and $u\in \pt S$,
which is sometimes very inconvenient.
However, Assumption~\ref{ass:actDefHen} is true
if $G(x)$ is {\it strictly matrix concave} on $\mc{D}$, that is,
for every $0\ne \xi \in \re^m$ the Hessian $-\nabla_{xx}(\xi^TG(x)\xi) \succ 0 $
for all $x\in \mc{D}$.
So, the following is a consequence of Theorem~\ref{thm:defHen}.

\begin{cor}  \label{cor:str-cav}
Suppose $S \subseteq int(\mc{D})$, $\mc{D}$ is convex, and
$G(\tilde{x})\succ 0$ for some $\tilde{x}\in \mc{D}$.
If $G(x)$ is strictly matrix concave on $\mc{D}$ and the archimedean condition holds,
then $S=L_N$ for all $N$ big enough.
\end{cor}

We now give an example of how to apply Theorem~\ref{thm:defHen} and Corollary~\ref{cor:str-cav}.

\begin{exm} \label{emp:square}
Consider $\mc{D}=[-1,1]^2$ is the square, $g_1(x)=1-x_1^2$, $g_2(x)=1-x_2^2$ and
\[
G(x) = \bbm 1-x_1^2-\half x_2^2  & \frac{1}{6}(x_1^3+x_2^3) \\
 \frac{1}{6}(x_1^3+x_2^3) & 1-\half x_1^2-x_2^2 \ebm.
\]
The matrix $G(x)$ is strictly concave over $\mc{D}$,
because for every $0\ne \xi \in\re^2$ the Hessian
\[
-\nabla_{xx}(\xi^TG(x)\xi) =
\bbm  2\xi_1^2+\xi_2^2-2\xi_1\xi_2 x_1 & 0 \\
0 & \xi_1^2+2\xi_2^2-2\xi_1\xi_2 x_2 \ebm
\]
is positive definite for all $x\in [-1,1]^2$.
So, the set $S=\{x\in [-1,1]^2: G(x)\succeq 0\}$ is convex.
Its boundary lies on the curve $\det G(x)=0$,
which is drawn of Figure~\ref{fig:square-strcav}.
The convex region surrounded by $\det G(x)=0$ is the set $S$,
which is drawn in the shaded area in Figure~\ref{fig:square-strcav}.
Some part of the curve $\det G(x)=0$ does not lie on the boundary $\pt S$,
because $G(x)$ is not positive semidefinite there.
\begin{figure}
\centering
\includegraphics[width=.66\textwidth]{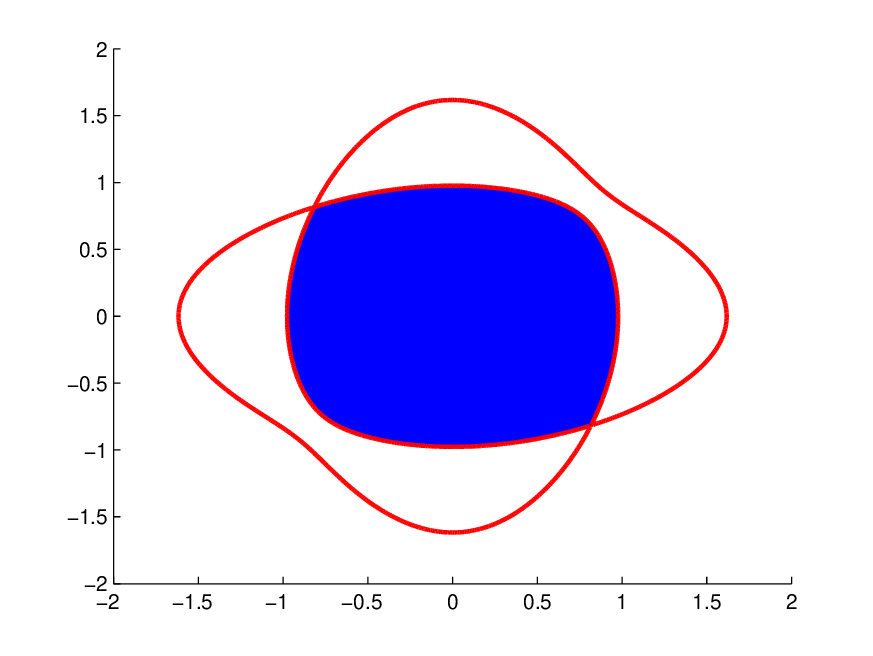}
\caption{ The shaded area is the convex set in Example~\ref{emp:square},
and the curve is $\det G(x)=0$.}
\label{fig:square-strcav}
\end{figure}
The determinant $\det G(x)$ is not concave
(also not convex) over $[-1,1]^2$,
e.g., its Hessian at $(0,3/4)$ is indefinite.

By Corollary~\ref{cor:str-cav}, the $L_N$ in \reff{def:L_N}
represents $S$ for $N$ big enough.
Actually, we have $L_2 = S$ in this example.
This justification would be obtained by investigating
the degree bound $N^*$ in \reff{eq:ballSOS}.
When $\Lambda = \xi\xi^T$ has rank one,
the matrix $H(u,x)$ defined in \reff{Tay:strct-cav} is
\[
\bbm
\half(2\xi_1^2+\xi_2^2) -\frac{1}{3}\xi_1\xi_2(2u_1+x_1)  & 0 \\
0 & \half (\xi_1^2+2\xi_2^2)-\frac{1}{3}\xi_1\xi_2(2u_2+x_2)
\ebm.
\]
Note the identities
\begin{align*}
\half(2\xi_1^2+\xi_2^2)-\frac{1}{3}\xi_1\xi_2(2u_1+x_1) & =
\frac{1}{3}\left(\xi_1 u_1 - \xi_2\right)^2 +
\frac{1}{6} \left( \xi_1 x_1 - \xi_2\right)^2 \\
& \quad + \frac{\xi_1^2}{2} +  \frac{\xi_1^2}{3}(1-u_1^2) + \frac{\xi_1^2}{6}(1-x_1^2), \\
\half(\xi_1^2+2\xi_2^2)-\frac{1}{3}\xi_1\xi_2(2u_2+x_2) & =
\frac{1}{3}\left(\xi_2 u_2 - \xi_1\right)^2 +
\frac{1}{6} \left(\xi_2 x_2 - \xi_1\right)^2 \\
& \quad + \frac{\xi_2^2}{2} +  \frac{\xi_2^2}{3}(1-u_2^2) + \frac{\xi_2^2}{6}(1-x_2^2).
\end{align*}
So, the representation of $H(u,x)$ in \reff{eq:ballSOS}
is true for $N^*=2$ when $\Lambda = \xi\xi^T$ has rank one.
Every $\Lambda \succeq 0$ is a sum of rank one matrices like $\xi\xi^T$.
Thus, the representation of $H(u,x)$ in \reff{eq:ballSOS} is also true
for $N^*=2$ when $\Lambda \succeq 0$ is general.
From the proof of Theorem~\ref{thm:defHen},
we can conclude that $L_2$ is a correct SDP representation of $S$.
\qed
\end{exm}

The construction of $L_N$ in \reff{def:L_N} is simply based on $g_i(x)$ and $G(x)$.
Theorem~\ref{thm:defHen} and Corollary~\ref{cor:str-cav}
tell us that $L_N$ is a correct SDP representation for $S$ for $N$ big enough,
when $G(x)$ is strictly concave over $\mc{D}$.
On the other hand, the degree bound $N$ is not given explicitly there,
which is not favorable in applications.
But the situation is not that bad in many cases.
As we have seen in Example~\ref{emp:square},
the degree bound $N$ would possibly be obtained
by investigating the representation of $H(u,x)$ in \reff{eq:ballSOS}.
We can assume $\Lambda$ is a rank one matrix,
and then use the strict matrix concavity of $G(x)$
to determine the degree bound $N^*$ in \reff{eq:ballSOS}.
This approach would work when $S$ is special like in Example~\ref{emp:square}.
For general case, the degree bound $N^*$ in \reff{eq:ballSOS}
is usually very difficult to get, as one would imagine.
But we have the same difficulty even when $G(x)$ is scalar.
In \cite{HN1}, when $S$ is defined by strictly concave polynomials,
it is only shown that the Lasserre type constructions would give a correct lifted LMI
when $N$ is big enough, but no explicit degree bounds are given there.
An interesting future work is to estimate good degree bounds.

We list some classes of $G(x)$ such that
the $N^*$ in \reff{eq:ballSOS} is relatively easy to estimate.
\bnum

\item Suppose $G(x)$ is of the form
\[
G(x) = A_0(x) + f_1(x)A_1 + \cdots + f_k(x) A_k
\]
where $A_0(x)$ is linear, every $f_i(x)$ is scalar and $A_1,\ldots,A_k \succeq 0$. Let
\[
H_{f_i}(u,x) = \int_0^1\int_0^t -\nabla_{xx} f_i(u+s(x-u)) ds dt.
\]
Then, the $H(u,x)$ in \reff{Tay:strct-cav} has the expression
\[
H(u,x) \, =  \, \sum_{i=1}^k H_{f_i}(u,x)  \Lambda \bullet A_i.
\]
Note that $\Lambda \bullet A_i \geq 0$ whenever $\Lambda \succeq 0$.
Clearly, if every $f_i(x)$ is strictly concave over $\mc{D}$,
then $G(x)$ is also strictly matrix concave over $\mc{D}$.
Therefore, the $N^*$ in \reff{eq:ballSOS}
would be investigated through studying the degree bound of the representation
\[
H_{f_i}(u,x) \, =  \, \sig_0(x) + g_1(x)\sig_1(x) + \cdots + g_r(x) \sig_r(x)
\]
where each $\sig_i(x)$ is sos.
This is relatively easier to do, because
the above $H_{f_i}(u,x)$ is independent of $\xi$.

\item Suppose $G(x)$ is of the form
\[
G(x) = A(x) + \diag (f_1(x), \ldots, f_m(x))
\]
where $A(x)$ is linear in $x$.
When $\Lambda = \xi\xi^T$ has rank one, we have the expression
\[
H(u,x) \, =  \, \sum_{i=1}^m \xi_i^2 H_{f_i}(u,x).
\]
The $G(x)$ is strictly matrix concave over $\mc{D}$ if and only if every $f_i(x)$ is so.
Thus, the $N^*$ in \reff{eq:ballSOS} would possibly be obtained
from estimating the degree bound for the representation of $H_{f_i}(u,x)$ as in the above.

\enum

We would like to remark that the smallest $L_d$ in \reff{def:L_N} is a lifted LMI for $S$
if $G(x)$ satisfies the q-module matrix concavity given in the next section.
This is a consequence of Theorem~\ref{thm:RatRep} and Corollary~\ref{cor:ufm-qmod}
that consider the more general case of $G(x)$ being rational.
This leads to our next section.

\section{Rational matrix inequality} \label{sec:rat-MI}
\setcounter{equation}{0}

This section assumes $S=\{x\in \mc{D}: G(x) \succeq 0\}$
is defined by a matrix rational function $G(x)$, i.e.,
every entry of $G(x)$ is rational.
Suppose $G(x)$ is matrix-concave on $\mc{D}$.
As before, the domain $\mc{D}=\{x\in\re^n: g_1(x) \geq 0, \ldots, g_r(x) \geq 0 \}$
is still defined by polynomials.
The case of $G(x)$ being a scalar rational function is discussed in \cite{Nie08}.
This section discusses the more general case of $G(x)$ being a matrix.
We first construct an SDP relaxation for $S$, and then prove
it represents $S$ when $G(x)$ satisfies certain conditions.

Suppose the matrix rational function $G(x)$ is given in the form
\be
G(x) = \frac{1}{ den(G(x)) } \sum_{\af \in \N^n: \, |\af| \leq \deg(G)} x^\af F_\af,
\ee
where $F_\af \in \re^{m\times m}$ are symmetric matrices,
$den(G(x))$ is the denominator of $G(x)$,
and $\deg(G)$ is the degree of $G(x)$,
which equals the maximum of degrees of the denominator and numerator.
%
%
Assume $den(G(x))$ is nonnegative on $\mc{D}$.
We say $G(x)$ is {\it q-module matrix concave} over $\mc{D}$
if for every $\xi\in \re^m$ there exist sos polynomials $\sig_{i,j}(x,u)$ such that
\be \label{qmod:G}
\baray{c}
den(G(x)) den(G(u))^2 \cdot \Big( \xi^TG(u)\xi + ( \nabla_x  \xi^TG(x)\xi) ^T\Big|_{x=u}(x-u)
- \xi^TG(x)\xi  \Big) = \\
\overset{m}{\underset{i=0}{\sum}}  g_i(x) \left(
\overset{m}{\underset{j=0}{\sum}} g_j(u) \sig_{ij}(x,u)\right)
\earay
\ee
is an identity in $(x,u)$. The above $g_0\equiv 1$.
The condition \reff{qmod:G} is based on Putinar's Positivstellensatz \cite{Putinar}.
Clearly, if $G(x)$ is q-module matrix concave over $\mc{D}$,
then it must also be matrix concave over $\mc{D}$.
We would like to remark that the $S$ considered in Section~3
is a special case here (the denominator $den (G)$ is $1$).

Now we turn to the construction of a lifted LMI for $S$. Assume $G(x)$
is q-module matrix concave over $\mc{D}$ and \reff{qmod:G} holds. Let
\be \label{deg:qm}
\baray{c}
d  = \max \left\{
\underset{0\leq i,j \leq r}{\max} \quad
\lceil \half \deg_x(g_i \sig_{ij}) \rceil, \half \deg(G) \right\}.
\earay
\ee
For $i=0,\ldots, r$, define matrices $P_\bt^{(i)}, Q_\af^{(i)}$ such that
\be \label{def:PQ}
\baray{rcl}
\frac{g_i(x)}{den(G(x))} [x]_{d-d_i}[x]_{d-d_i}^T &= &
\underset{\af \in \N^n: |\af| + |LE(den(G))| \leq 2d }{\sum} Q_\af^{(i)} x^\af
+ \underset{\bt\in\N_{\leq 2d}: \bt < LE(den(G))}{\sum} P_\bt^{(i)} \frac{x^\bt}{den(G(x))}.
\earay
\ee
Here $LE(p)$ denotes the exponent of the leading monomial of $p(x)$
in the lexicographical ordering ($x_1 > x_2  > \cdots > x_n $),
that is, $x^\af > x^\bt$ if the first nonzero entry of $\af-\bt$ is bigger than $0$.
Let $y$ be a vector indexed by $\af \in \N^n$ with $|\af| + |LE(den(G))| \leq 2d$,
and $z$ be a vector indexed by $\bt\in\N_{\leq 2d}$ with $\bt < LE(den(G))$.
Define linear matrix pencils
\be  \label{lmi:Qyz}
\baray{l}
Q_i(y,z) =
\underset{\af \in \N^n: |\af| + |LE(den(G))| \leq 2d }{\sum} Q_\af^{(i)} y_\af
+ \underset{\bt\in\N_{\leq 2d}: \bt < LE(den(G))}{\sum} P_\bt^{(i)}  z_\bt,
\quad \, i= 0,1,\ldots,r.
\earay
\ee
Here $P_\af^{(i)}, Q_\af^{(i)}$ are defined in \reff{def:PQ}.
Suppose $G(x)$ can be written as
\[
G(x) =
\sum_{\af \in \N^n: |\af| + |LE(den(G))| \leq 2d } F_\af^{(1)} x^\af
+ \sum_{\bt\in\N_{\leq 2d}: \bt < LE(den(G))} F_\bt^{(2)} \frac{x^\bt}{den(G(x))},
\]
then define the linear matrix pencil
\[
F(y,z) =
\sum_{\af \in \N^n: |\af| + |LE(den(G))| \leq 2d } F_\af^{(1)} y_\af
+ \sum_{\bt\in\N_{\leq 2d}: \bt < LE(den(G))} F_\bt^{(2)} z_\bt.
\]
Clearly, $S$ can be equivalently described as
\[
S = \left\{(y_{e_1},\ldots,y_{e_n})
\left|\baray{l}
\exists\, x \in \re^n, \, y_\af = x^\af, z_\bt = \frac{x^\bt}{den(G(x))}
\quad \forall \, \af,\bt  \\
F(y, z) \succeq 0,\,
Q_i(y,z) \succeq 0, \, i=0,\ldots,r
\earay \right.
\right\}.
\]
If we remove $y_\af = x^\af, z_\bt = \frac{x^\bt}{den(G(x))}$ in the above,
then $S$ is a subset of
\be \label{Lqmod}
L_{qmod}  = \left\{x\in \re^n
\left|\baray{l}
\exists \,y,z, \, y_{0} = 1, x_1=y_{e_1}, \ldots, x_n = y_{e_n}, \\
F(y, z) \succeq 0,\, Q_i(y,z) \succeq 0, \, i=0,\ldots,r
\earay\right.
\right\}.
\ee
So, $S\subseteq L_{qmod}$.
We are interested in conditions guaranteeing $S = L_{qmod}$.

\begin{lem}  \label{lem:RATlinrep}
Assume $S\subset int(\mc{D})$ and $G(\tilde{x})\succ 0$
for some $\tilde{x}\in \mc{D}$.
Suppose $G(x)$ is q-module matrix concave over $\mc{D}$.
If $v\in \pt S$, $den(G(v))>0$,
and $a^T(x-v) \geq 0$ for all $x\in S$,  then
\[
den(G(x)) \cdot (a^T (x-v)  - \Lambda \bullet G(x)  )=  \sum_{i=0}^r  g_i(x) \sig_i(x)
\]
for some $\Lambda \succeq 0$ and sos polynomials $\sig_i$
with $\deg(g_i\sig_i) \leq 2d$.
\end{lem}
\begin{proof}
Consider the linear optimization problem
\[
\min_{x\in int(\mc{D})}  \quad a^Tx \quad
\mbox{ subject to } \quad  G(x)\succeq 0.
\]
The point $v \in \pt S$ is an optimizer.
Since $den(G(v))>0$, $G(x)$ is differentiable at $v$.
Since $S\subset int(\mc{D})$, $v$ is in the interior of $\mc{D}$.
Because $G(\tilde{x})\succ 0$
(the Slater's condition holds) and $G(x)$ is matrix concave on $\mc{D}$,
there exists $\Lambda \succeq 0$ such that
(see \cite[p.~306]{Shp})
\begin{align*}
a =  \nabla_x (\Lambda \bullet G(v)), \quad
\Lambda \bullet G(v) = 0.
\end{align*}
Hence, we get the identity
\begin{align*}
a^T(x-v) - \Lambda \bullet G(x)  =
\Lambda  \bullet  G(v) + \nabla_x (\Lambda \bullet G(v))^T(x-v)  - \Lambda \bullet G(x).
\end{align*}
Since $\Lambda \succeq 0$, we have a decomposition
$\Lambda = \sum_{i=1}^K \lmd^{(i)} (\lmd^{(i)})^T$. Then, it holds that
\begin{align*}
 & a^T(x-v) - \Lambda \bullet G(x) =  \\
\sum_{i=1}^K \Big\{
(\lmd^{(i)})^TG(v) \lmd^{(i)} + & \nabla_x ((\lmd^{(i)})^T G(v) \lmd^{(i)})^T(x-v)  -
(\lmd^{(i)})^TG(x) \lmd^{(i)} \Big\}.
\end{align*}
So, the lemma readily follows the q-module matrix concavity of $G(x)$.
\end{proof}

For a function $f(x)$, denote by $\mc{Z}(f)$ its real zero set, i.e.,
$\mc{Z}(f)=\{x\in\re^n: f(x) =0\}$.

\begin{theorem}  \label{thm:RatRep}
Assume $S$ is closed and convex, $S\subset int(\mc{D})$,
$G(\tilde{x})\succ 0$ for some $\tilde{x}\in \mc{D}$, and
\[
\dim\big( \, \mc{Z}(den(G)) \cap \pt S \, \big) < n-1.
\]
If $G(x)$ is q-module matrix concave over $\mc{D}$,
then $S = L_{qmod}$.
\end{theorem}
\begin{proof}
Since $S\subseteq L_{qmod}$, it is sufficient for us to prove the reverse containment.
By a contradiction proof, suppose otherwise there exists $\hat{x} \in L_{qmod}/S$
and $(\hat{y},\hat{z})$ such that
\[
\hat{x}=(\hat{y}_{e_1},\ldots,\hat{y}_{e_n}), \quad
F(\hat{y},\hat{z}) \succeq 0,\quad
Q_i(\hat{y},\hat{z}) \succeq 0, \, i=0,\ldots,r.
\]
Since $S$ is convex and closed,
by Hahn-Banach Theorem, there exists a supporting hyperplane
$\{a^Tx = b\}$ of $S$ such that
$a^Tx \geq b$ for all $x\in S$ and $a^T \hat x < b$.
Let $v \in \pt S$ be a minimizer of $a^Tx$ on $S$.
Since $\dim\big( \mc{Z}(den(G)) \cap \pt S\big) < n-1$, by continuity,
the supporting hyperplane $\{a^Tx = b\}$
can be chosen to satisfy $den(G(v)) > 0$.
By Lemma~\ref{lem:RATlinrep}, we have
\be \label{eq:Rlrep2}
a^T(x-v) =\Lambda \bullet  G(x)+ \sum_{i=0}^r \frac{g_i(x)}{den(G(x))} \sig_i(x)
\ee
for some sos polynomials $\sig_i$ such that every $\deg(g_i \sig_i) \leq 2d$.
If we write $\sig_i$ as
\[
\sig_i(x) = [x]_{d-d_i}^T W_i [x]_{d-d_i}
\]
for symmetric $W_i \succeq 0 \, (i=0,1,\ldots,r)$, then the identity \reff{eq:Rlrep2} becomes
\begin{align*}
a^T(x-v)  =\Lambda \bullet G(x)+ \sum_{i=0}^r
\left(\frac{g_i(x)}{den(G(x))} [x]_{d-d_i} [x]_{d-d_i}^T\right) \bullet W_i =
\Lambda \bullet G(x) + \\
\sum_{i=0}^r  \left(
\underset{\af \in \N^n: |\af| + |LE(den(G))| \leq 2d }{\sum}  Q_\af^{(i)} x^\af
+ \underset{\bt \in \N^n: \bt < LE(den(G))}{\sum} P_\af^{(i)}  \frac{x^\bt}{den(G(x))} \right) \bullet W_i.
\end{align*}
In the above identity, if we replace every $x^\af$ by $\hat y_\af$
and $\frac{x^\bt}{den(G(x))}$ by $\hat z_\bt$, then
\[
a^T \hat x - b = \Lambda \bullet F(\hat{y}, \hat{z}) +
\sum_{i=0}^r Q_i(\hat y, \hat z) \bullet W_i \geq 0,
\]
because all $ \Lambda, F(\hat{y}, \hat{z}), Q_i(\hat y, \hat z), W_i \succeq 0$.
This contradicts $a^T\hat{x}<b$. So, $S=L_{qmod}$.
\end{proof}

The condition of q-module matrix concavity requires checking \reff{qmod:G}
for every $\xi\in\re^m$. In many situations this is almost impossible.
However, if we consider $\xi$ as an indeterminant, then a sufficient condition
guaranteeing \reff{qmod:G} is
\be \label{qmod:G-xi}
\baray{c}
den(G(x)) den(G(u))^2 \cdot \Big( \xi^TG(u)\xi + ( \nabla_x  \xi^TG(x)\xi)^T\Big|_{x=u}
(x-u) - \xi^TG(x)\xi  \Big) = \\
\overset{m}{\underset{i=0}{\sum}}  g_i(x) \left(
\overset{m}{\underset{j=0}{\sum}} g_j(u) \sig_{ij}(x,u,\xi)\right),
\earay
\ee
where every $\sig_{ij}(x,u,\xi)$ is now an sos polynomial in $(x,u,\xi)$.
If $G(x)$ satisfies \reff{qmod:G-xi},
we say $G(x)$ is {\it uniformly q-module matrix concave} over $\mc{D}$.
Clearly, the corollary below follows Theorem~\ref{thm:RatRep}.

\begin{cor}  \label{cor:ufm-qmod}
Assume $S$ is closed and convex, $S\subset int(\mc{D})$,
$G(\tilde{x})\succ 0$ for some $\tilde{x}\in \mc{D}$, and
\[
\dim\big( \, \mc{Z}(den(G)) \cap \pt S \, \big) < n-1 .
\]
If $G(x)$ is uniformly q-module matrix concave over $\mc{D}$,
then $S = L_{qmod}$.
\end{cor}

We would like to remark that the $L_{qmod}$ in \reff{Lqmod} is equivalent to
the $L_N$ in \reff{def:L_N} for $N=d$
when $G(x)$ is a matrix polynomial (its denominator is $1$).
Therefore, Theorem~\ref{thm:RatRep} and Corollary~\ref{cor:ufm-qmod}
imply that $L_d$ defined in \reff{def:L_N}
is also a correct SDP representation for $S$
under the (uniform) q-module matrix concavity.

Now we give some examples on how to apply
Theorem~\ref{thm:RatRep} and Corollary~\ref{cor:ufm-qmod}.

\begin{exm} \label{exp:qmod_orth}
Consider the set $S=\left\{x\in\re_+^2: G(x)\succeq 0\right\}$ where
\[
G(x) = \bbm 7-x_1+2x_2 & 5 \\ 5  & 11-x_2 \ebm
-\frac{1}{x_1x_2}
\bbm x_1 + x_2^3 & x_2^2 \\ x_2^2 & x_2 \ebm.
\]
Its domain $\mc{D}=\re_+^2$. The determinant of $G(x)$ is
\[
\frac{1}{x_1x_2}(x_1^2x_2^2 - 11x_1^2x_2 - 2x_1x_2^3 + 15x_1x_2^2 + 54x_1x_2
- 11x_1 + x_2^4 - 11x_2^3 + 8x_2^2 - 7x_2 + 1).
\]
Clearly, the boundary $\pt S$ lies on the curve $\det(G(x))=0$.
It is a planar curve of degree $4$,
and is drawn in Figure~\ref{fig:qmod_orthant}.
\begin{figure}
\centering
\includegraphics[width=.66\textwidth]{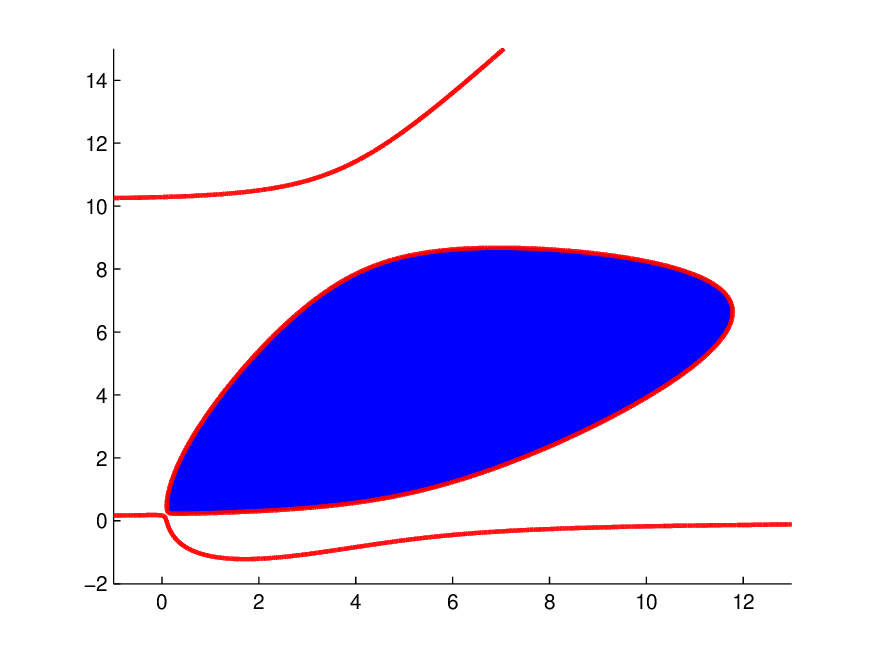}
\caption{ The shaded area is the convex set in Example~\ref{exp:qmod_orth},
and the curve is $\det G(x)=0$.}
\label{fig:qmod_orthant}
\end{figure}
The $G(x)$ here is uniformly q-module matrix concave over $\re_+^2$, because
\be \nn
\baray{c}
x_1x_2u_1^2u_2^2 \cdot \Big( \xi^TG(u)\xi +
( \nabla_x  \xi^TG(x)\xi) ^T\Big|_{x=u}(x-u) - \xi^TG(x)\xi  \Big) = \\
x_2u_2^2(u_1\xi_2 - x_1\xi_2 + u_1x_2\xi_1 - u_2x_1\xi_1)^2
+x_1u_1^2\xi_1^2(u_2 - x_2)^2.
\earay
\ee
Hence, the set $S$ here is convex, and by Corollary~\ref{cor:ufm-qmod}
it has the following lifted LMI
\[
\bbm 7-x_1+2x_2 & 5 \\ 5  & 11-x_2 \ebm -
\bbm z_{10} + z_{03} & z_{02} \\ z_{02} & z_{01} \ebm \succeq 0,
\]
\[
\bbm
0 & 0 & 0 & 0 & 1 & 0 \\
0 & 0 & 1 & 0 & x_1 & x_2 \\
0 & 1 & 0 & x_1 & x_2 & 0 \\
0 & 0 & x_1 & 0 & y_{20} & y_{11} \\
1 & x_1 & x_2 & y_{20} & y_{11} & y_{02} \\
0 & x_2 & 0 & y_{11} & y_{02} & 0
\ebm +
\bbm
z_{00} & z_{10} & z_{01} & z_{20} & 0 & z_{02} \\
z_{10} & z_{20} & 0 & z_{30} & 0 & 0 \\
z_{01} & 0 & z_{02} & 0 & 0 & z_{03} \\
z_{20} & z_{30} & 0 & z_{40} & 0 & 0 \\
0 & 0 & 0 & 0 & 0 & 0 \\
z_{02} & 0 & z_{03} & 0 & 0 & z_{04}
\ebm \succeq 0,
\]
\[
\bbm
0 & 0 & 1 \\
0 & 0 & x_1 \\
1 & x_1 & x_2
\ebm +
\bbm
z_{10} & z_{20} & 0 \\
z_{20} & z_{30} & 0 \\
0 & 0 & 0
\ebm \succeq 0,
\bbm
0 & 1 & 0 \\
1 & x_1 & x_2 \\
0 & x_2 & 0
\ebm +
\bbm
z_{01} & 0 & z_{02} \\
0 & 0 & 0 \\
z_{02} & 0 & z_{03}
\ebm \succeq 0.
\]
The set of $x$ satisfying the above LMIs
is drawn in the shaded area of Figure~\ref{fig:qmod_orthant}.
The convex region there surrounded by $\det G(x)=0$
is the set $S$, which is precisely the shaded area.
Some components of the curve $\det G(x)=0$ do not lie on the boundary $\pt S$,
because $G(x)\succeq 0$ fails there.
This confirms the above lifted LMI is correct.
\qed
\end{exm}

%
%

\begin{exm} \label{exmp:qmod_plane}
Consider the set $S=\left\{x\in\re^2: G(x) \succeq 0\right\}$ where
\[
G(x) = \bbm  1-2x_1^2-2x_1x_2-x_2^2 & x_1^2 \\ x_1^2  & 1-x_1^2 \ebm
+\frac{x_2^4}{x_1^2+x_2^2}
\bbm  -1 & 1 \\ 1 & -1 \ebm^T.
\]
Clearly, the boundary of $S$ lies on the curve $\det G(x)=0$,
which is drawn in Figure~\ref{fig:qmod_plane_deg4}. It has three connected components.
\begin{figure}
\centering
\includegraphics[width=.66\textwidth]{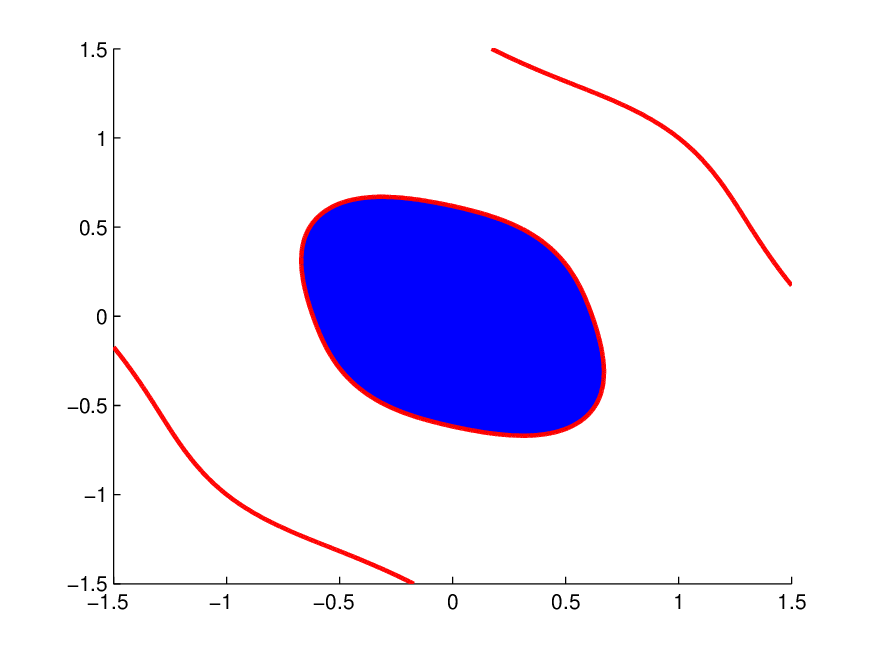}
\caption{ The shaded area is the convex set in Example~\ref{exmp:qmod_plane},
and the curve is $\det G(x) =0$.}
\label{fig:qmod_plane_deg4}
\end{figure}
The domain $\mc{D}=\re^2$, and the above
$G(x)$ is uniformly q-module matrix concave over $\re^2$, because
\be \nn
\baray{c}
\|x\|_2^2\cdot \|u\|_2^4 \cdot \Big( \xi^TG(u)\xi +
( \nabla_x  \xi^TG(x)\xi) ^T\Big|_{x=u}(x-u) - \xi^TG(x)\xi  \Big) = \\
\left(\overset{9}{\underset{i=1}{\sum}} f_i^2\right) \cdot (\xi_1-\xi_2)^2
+ \half \|x\|_2^2\cdot \|u\|_2^4 \cdot \|x-u\|_2^2\cdot \xi_1^2
\earay
\ee
where the polynomials $f_i$ are given as below
\begin{align*}
\baray{ll}
f_1 = -u_1 u_2 x_2^2-u_1 u_2 x_1^2+u_1 u_2^2 x_2+u_1^2 u_2 x_1, &
f_6 = \frac{1}{\sqrt{2}}(-u_2^2 x_2^2+u_2^3 x_2-u_1^2 x_1^2+u_1^3 x_1),\\
f_2 = -u_1 u_2 x_2^2+u_1 u_2 x_1^2+u_1 u_2^2 x_2-u_1^2 u_2 x_1, &
f_7 = -2 u_1 u_2 x_1 x_2 + u_1 u_2^2 x_1 +  u_1^2 u_2 x_2,\\
f_3 = \frac{1}{\sqrt{2}}(-u_2^2 x_1 x_2+u_2^3 x_1-u_1^2 x_1 x_2+u_1^3 x_2), &
f_8 = u_2^2 x_1^2 - u_1^2 x_2^2, \\
f_4 = \frac{1}{\sqrt{2}}( u_2^2 x_1 x_2-u_2^3 x_1-u_1^2 x_1 x_2+u_1^3 x_2), &
f_9 = -u_1 u_2^2 x_1+u_1^2 u_2 x_2. \\
f_5 =  \frac{1}{\sqrt{2}}(u_2^2 x_2^2-u_2^3 x_2-u_1^2 x_1^2+u_1^3 x_1),\\
\earay
\end{align*}
So, the set $S$ is convex, and by Corollary~\ref{cor:ufm-qmod}
a lifted LMI for it is
\[
\bbm  1-2y_{20}-2y_{11}-y_{02}-z_{04} &  y_{20}+z_{04} \\
y_{20}+z_{04} & 1-y_{20}-z_{04} \ebm \succeq 0,
\]
\[
\bbm
0 & 0   & 0   & 1      & 0      & 0 \\
0 & 1   & 0   & x_1    & x_2    & 0 \\
0 & 0   & 0   & x_2    & 0      & 0 \\
1 & x_1 & x_2 & y_{20}-y_{02} & y_{11} & y_{02} \\
0 & x_2 & 0   & y_{11} & y_{02}  & 0 \\
0 & 0   & 0   & y_{02} & 0      & 0
\ebm +
\bbm
z_{00}  &  z_{10} & z_{01}  & -z_{02}  & z_{11}  & z_{02} \\
z_{10}  & -z_{02} & z_{11}  & -z_{12}  & -z_{03} & z_{12} \\
z_{01}  & z_{11}  & z_{02}  & -z_{03}  & z_{12}  & z_{03} \\
-z_{02} & -z_{12} & -z_{03} & z_{04}  & -z_{13} & -z_{04} \\
z_{11}  & -z_{03} & z_{12}  & -z_{13}  & -z_{04} & z_{13} \\
z_{02}  & z_{12}  & z_{03}  & -z_{04}  & z_{13}  & z_{04}
\ebm \succeq 0.
\]
The feasible points $x$ satisfying the above
are drawn
in the shaded area of Figure~\ref{fig:qmod_plane_deg4}.
Only one component of the curve $\det G(x)=0$ lies on the boundary $\pt S$.
The other two do not touch $S$ since $G(x)\succeq 0$ fails there.
This confirms that the above LMI represents $S$.
\end{exm}

Now we make some remarks about the condition of (uniform) q-module matrix concavity.
When the denominator $den(G(x))$ is nonnegative on $\mc{D}$,
$G(x)\succeq 0$ is inequivalent to $den(G(x))\cdot G(x) \succeq 0$,
and the latter is a polynomial matrix inequality.
Thus, one would reduce a rational matrix inequality
to a polynomial matrix inequality and then apply the results
of Sections~\ref{sec:mat-sos} and \ref{sec:strct-cocav}.
However, we would like to point out that multiplying $den(G(x))$
usually destroys the matrix concavity of $G(x)$,
and the resulting $den(G(x))\cdot  G(x)$ typically is not
matrix concave or (uniform) matrix sos-concave.
In Examples~\ref{exp:qmod_orth} and \ref{exmp:qmod_plane},
it would be easily verified that the
$den(G(x))\cdot  G(x)$'s there are not matrix concave.

We now list some classes of rational $G(x)$
that are (uniformly) q-module matrix concave.

\bnum

\item Suppose $G(x)$ is of the form
\[
G(x) \, = \, A_0(x)+f_1(x)A_1+\cdots+f_k(x)A_k
\]
where $A_0(x)$ is linear,
$f_1(x),\ldots, f_k(x)$ are q-module concave rational functions over $\mc{D}$,
and $A_1,\ldots,A_k \succeq 0$.
Then $G(x)$ is clearly q-module concave over $\mc{D}$.

\item Suppose $G(x)$ is of the form
\[
G(x) \, = \, F(x)+\diag(f_1(x),\ldots, f_m(x))
\]
where $F(x)$ is linear and each $f_i(x)$ is a scalar rational function.
Clearly, $G(x)$ is matrix concave over $\mc{D}$ if and only if
every $g_i(x)$ is so,
and $G(x)$ is (uniformly) q-module matrix concave if and only if
every $g_i(x)$ is so.

\enum

\section{Conclusions}

This paper gives explicit constructions of SDP representations for
the set $S=\{x\in\mc{D}: G(x)\succeq 0\}$
when $G(x)$ is a matrix polynomial or rational function,
and proves sufficient conditions justifying them.
These conditions are based on the matrix concavity of $G(x)$.

We would like to remark that the SDP relaxations
\reff{def:L-sos} and \reff{def:L_N}
would be tightened if we replace $\mc{G}(y)\succeq 0$
by a bigger LMI. This follows an construction introduced
by Henrion and Lasserre \cite[II.D.]{HL06}.
Note that $G(x)\succeq0$ is equivalent to the PMI
(use $\otimes$ to denote Kronecker product of matrices)
\[
G(x) \otimes [x]_k[x]_k^T \succeq 0.
\]
The basic idea of their construction is that
replacing every monomial $x^\af$ in the expansion of $G(x) \otimes [x]_k[x]_k^T$
by a linear moment $y_\af$. Then, one would get a bigger LMI, say, $\mc{K}(y)\succeq 0$.
Since $G(x)$ is the first block of $G(x) \otimes [x]_k[x]_k^T$,
$\mc{G}(y)$ is a leading principle submatrix of $\mc{K}(y)$.
Thus, the LMI $\mc{K}(y)\succeq 0$ is tighter than $\mc{G}(y)\succeq 0$
in relaxing the set $S$. Therefore, if the constructions
\reff{def:L-sos} and \reff{def:L_N} use $\mc{K}(y)\succeq 0$
instead of $\mc{G}(y)\succeq 0$,
we can get similar semidefinite representability results
like Theorems~\ref{thm:mat-sos} and \ref{thm:defHen}.
On the other hand, the LMI $\mc{G}(y)\succeq 0$ is simpler than $\mc{K}(y)\succeq 0$,
and hence might be preferable in applications.
Under the archimedean condition, Henrion and Lasserre \cite{HL06} proved
the asymptotic convergence of the hierarchy of SDP relaxations
for minimizing a polynomial function subject to $G(x)\succeq 0$.

The matrix concavity is a strong condition for $S$ to be convex.
Generally it is very difficult to check.
A stronger but relatively easier checkable one is matrix sos-concavity.
This condition would also be difficult to check,
e.g., for the quadratic case it is already NP-hard.
A further stronger but much easier checkable condition is the uniform matrix sos-concavity,
which would be done by solving a single SDP.
When $G(x)$ is rational, similar conditions are (uniform) q-module matrix concavity.
Under these conditions, we justified some explicit SDP representations for $S$.
These conditions are certainly very strong.
However, to the author's best knowledge,
there are no more general conditions than them
for justifying the lifted LMIs constructed in this paper.
An interesting future work is to seek weaker conditions
justifying some efficiently constructible SDP representations.

\bigskip
\noindent
{\bf Acknowledgement} \,
The author would like to thank Bill Helton and the anonymous referees for fruitful suggestions
on improving this paper.


\begin{thebibliography}{99}

\bibitem{AP09}
A. A. Ahmadi and P.A. Parrilo.
A convex polynomial that is not sos-convex.
{\it Preprint}, 2009. \url{http://arxiv.org/abs/0903.1287}


\bibitem{BTN}
A.~Ben-Tal and A.~Nemirovski.
{\it Lectures on Modern Convex Optimization:
Analysis, Algorithms, and Engineering Applications}.
MPS-SIAM Series on Optimization,
SIAM, Philadelphia, 2001.


%

%


\bibitem{Choi}
M.-D. Choi.
Positive semidefinite biquadratic forms.
{\it Linear Algebra and Applications}, 12 (1975), pp. 95-100.


\bibitem{CLR}
M. D. Choi, T.-Y. Lam, and B. Reznick.
Real zeros of positive semidefinite forms.
{\it I. Math. Z.}, 171(1):1-26, 1980.


%


%
%

%


\bibitem{HN1}
J.W. Helton and J. Nie. Semidefinite representation of convex sets.
{\it Mathematical Programming}, Ser.~A, Vol. 122, No.1, pp.21--64, 2010.


\bibitem{HN2}
J.W. Helton and J. Nie.
Sufficient and necessary conditions for semidefinite representability of convex hulls and sets.
{\it SIAM Journal on Optimization}, Vol. 20, No.2, pp. 759-791, 2009.


\bibitem{HN08}
J.W. Helton and J. Nie.
Structured semidefinite representation of some convex sets.
{\it Proceedings of 47th IEEE Conference on Decision and Control},
pp. 4797 - 4800, Cancun, Mexico, Dec. 9-11, 2008.


\bibitem{HV}
W.~Helton and V.~Vinnikov.
Linear matrix inequality representation of sets.
{\em Comm.~Pure Appl.~Math.}  60 (2007), No. 5, pp. 654-674.



\bibitem{HL06}
D.~Henrion and J.~Lasserre.
Convergent relaxations of polynomial matrix inequalities and static output feedback.
{\it IEEE Trans. Auto. Control},\, {\bf 51}:192--202, 2006.


\bibitem{Hen08}
D. Henrion.
On semidefinite representations of plane quartics.
LAAS-CNRS Research Report No. 08444, September 2008.


%



\bibitem{Las06}
J.~Lasserre. Convex sets with semidefinite representation.
{\it Mathematical Programming}, Vol.~120, No.~2, pp. 457--477, 2009.

\bibitem{Las08}
J.~Lasserre.
Convexity in semi-algebraic geometry and polynomial optimization.
{\it SIAM Journal on Optimization}, Vol.~19, No.~4, pp. 1995 -- 2014, 2009.



\bibitem{Lau07}
M.~Laurent.
Semidefinite representations for finite varieties.
{\it Math. Program.} 109 (2007),  no. 1, Ser. A, 1--26.


\bibitem{LNQY}
C. Ling, J. Nie, L. Qi, and Y. Ye.
Bi-quadratic optimization over unit spheres and semidefinite programming relaxations.
{\it SIAM Journal on Optimization}, Vol. 20, No.~3, pp. 1286-1310, 2009.


\bibitem{yalmip}
J. L\"{o}fberg. YALMIP: a toolbox for modeling and optimization in Matlab.
{\it Proc. IEEE CACSD Symposium}, Taiwan, 2004.
\url{www.control.isy.liu.se/~johanl}




\bibitem{NN94}
Y.~Nesterov and A.~Nemirovskii. Interior-point polynomial
algorithms in convex programming. SIAM Studies in Applied
Mathematics, 13. Society for Industrial and Applied Mathematics
(SIAM), Philadelphia, PA, 1994. 


\bibitem{N06}
A.~Nemirovskii. Advances in convex optimization: conic
programming. {\it Plenary Lecture,  International Congress of
Mathematicians (ICM)}, Madrid, Spain, 2006.


\bibitem{Nie08}
J.~Nie.
First order conditions for semidefinite representations of convex sets defined by
rational or singular polynomials.
{\it Mathematical Programming}, to appear.



\bibitem{NiSt09}
J.~Nie and B.~Sturmfels.
Matrix cubes parametrized by eigenvalues.
{\it SIAM Journal on Matrix Analysis and Applications},
Vol. 31, No. 2, pp. 755-766, 2009.





\bibitem{Par06}
P.~Parrilo. Exact semidefinite representation for genus zero
curves. Talk at the Banff workshop ``Positive Polynomials and
Optimization'', Banff, Canada, October 8-12, 2006.


%

%


\bibitem{Putinar}
M.~Putinar. Positive polynomials on compact semi-algebraic sets,
{\it  Ind. Univ. Math. J.} \,  42 (1993) 203--206.

%


%



\bibitem{Shp}
A.~Shapiro.
First and second order analysis of nonlinear semidefinite programs.
{\it Mathematical Programming} 77 (1997), no. 2, Ser. B, 301--320.







\bibitem{VB}
L. Vandenberghe and S. Boyd.
Semidefinite programming.
{\it SIAM Review} 38 (1996), pp. 49-95.

\bibitem{SDPhb}
H. Wolkowicz, R. Saigal, and L. Vandenberghe, editors. {\it
Handbook of semidefinite programming.} Kluwer's Publisher, 2000.


\end{thebibliography}
\end{document}